\documentclass[12pt]{amsart}
\usepackage[nobysame,abbrev,alphabetic]{amsrefs}
\usepackage{amssymb}
\usepackage[all]{xy}
\usepackage{stmaryrd}

\DeclareMathOperator{\Img}{Im}
\DeclareMathOperator{\Inf}{Inf}
\DeclareMathOperator{\Ker}{Ker}
\DeclareMathOperator{\Res}{Res}
\DeclareMathOperator{\Span}{Span}
\DeclareMathOperator{\Supp}{Supp}
\DeclareMathOperator{\trg}{trg}
\newcommand{\isom}{\cong}

\newtheorem{thm}{Theorem}[section]
\newtheorem{cor}[thm]{Corollary}
\newtheorem{lem}[thm]{Lemma}
\newtheorem{prop}[thm]{Proposition}
\newtheorem{defin}[thm]{Definition}
\newtheorem{exam}[thm]{Example}

\newtheorem{rem}[thm]{Remark}
\newtheorem*{mainthm}{Main Theorem}
\swapnumbers

\numberwithin{equation}{section}

\begin{document}

\title[Cofinality of Galois Cohomology]{Cofinality of Galois Cohomology within Purely Quadratic Graded Algebras}

\author{Tamar Bar-On}
\address{Department of Mathematics\\
Ben-Gurion University of the Negev\\
P.O.\ Box 653, Be'er-Sheva 84105\\
Israel}
\curraddr{Jerusalem College of Technology\\
Hava'ad Haleumi 21, Jerusalem 91160\\
Israel} 
\email{tamarnachshoni@gmail.com}

\author{Ido Efrat}
\address{Earl Katz Family Chair in Pure Mathematics\\
Department of Mathematics\\
Ben-Gurion University of the Negev\\
P.O.\ Box 653, Be'er-Sheva 84105\\
Israel} \email{efrat@bgu.ac.il}

\thanks{This work was supported by the Israel Science Foundation (grant No.\ 569/21).
The first author was supported by the BGU Center for Advanced Studies in Mathematics}



\maketitle
\vspace*{-10pt}

\begin{abstract}
Let $p$ be a prime number.
For a field $F$ containing a root of unity of order $p$, let $H^\bullet(F)=H^\bullet(F,\mathbb{F}_p)$ be the mod-$p$ Galois cohomology graded $\mathbb{F}_p$-algebra of $F$.
By the Norm Residue Theorem, $H^\bullet(F)$ is a purely quadratic graded-commutative algebra, and is therefore determined by the cup product $\cup\colon H^1(F)\times H^1(F)\to H^2(F)$.
We prove that the class of all Galois cohomology algebras $H^\bullet(F)$ is cofinal in the class of all purely quadratic graded-commutative $\mathbb{F}_p$-algebras $A_\bullet$, in the following sense:
For every $A_\bullet$ there exists $F$ such that the bilinear map $A_1\times A_1\to A_2$, which determines $A_\bullet$, embeds in the cup product bilinear map $\cup\colon H^1(F)\times H^1(F)\to H^2(F)$.

We further provide examples of $\mathbb{F}_p$-bilinear maps which are not realizable by fields $F$ in this way.
These are related to recent results by Snopce--Zalesskii and Blumer-Quadrelli--Weigel on the Galois theory of pro-$p$ right-angled Artin groups, as well as to a conjecture by Marshall on the possible axiomatization of quadratic form theory of fields.  
\end{abstract}

\tableofcontents

\section{Introduction}
We fix a prime number $p$.
Let $F$ be a field containing a root of unity $\zeta$ of order $p$ and let $G_F={\rm Gal}(F^{\rm sep}/F)$ be its absolute Galois group.
We write $H^i(F)$ for the $i$th profinite cohomology group $H^i(G_F,\mathbb{F}_p)$, where $G_F$ acts trivially on $\mathbb{F}_p$, and consider the cohomology $\mathbb{F}_p$-algebra $H^\bullet(F)=\bigoplus_{i\geq0}H^i(F)$ with the cup product.
A fundamental property of this graded-commutative algebra is given by the celebrated \textsl{Norm Residue Theorem} (formerly the Bloch--Kato conjecture; See \S\ref{section on Galois cohomology}), proved by Voevodsky and Rost with a contribution by Weibel:
$H^\bullet(F)$ is a \textsl{quadratic} graded algebra, which loosely speaking means that it is generated by its degree $1$ elements, with relations originating from the degree 2 component.
Moreover, it is \textsl{purely quadratic}, meaning that the defining relations can be taken to be pure tensors in $H^1(F)^{\otimes2}$  (see \S\ref{section on kappa-algebras} for the precise definition).
A major problem in Galois cohomology is to characterize the graded $\mathbb{F}_p$-algebras that arise as Galois cohomology algebras $H^\bullet(F)$ among all purely-quadratic graded-commutative $\mathbb{F}_p$-algebras.

The graded algebra $H^\bullet(F)$ has the additional property that there exists a distinguished element $\epsilon\in H^1(F)$ such that $2\epsilon=0$ and $\chi\cup(\epsilon+\chi)=0$ for every $\chi\in H^1(F)$.
Specifically, $\epsilon=(-1)_F$ is the Kummer element of $-1$, and this identity is the cohomological interpretation of the splitting of symbol $F$-algebras $(a,-a)_{F,\zeta}$ (see \S\ref{section on Galois cohomology}). 
Therefore, we restrict ourselves to purely quadratic graded-commutative $\mathbb{F}_p$-algebras $A_\bullet=\bigoplus_{i\geq0}A_i$ with a distinguished element $\epsilon\in A_1$ satisfying these two conditions.
Thus $(A_\bullet,\epsilon)$ is a \textsl{$\kappa$-algebra} in the sense of Bass--Tate \cite{BassTate73}.
Note that in the case where $p$ is odd we have $\epsilon=0$, so the structures $A_\bullet$ are just the purely quadratic graded-commutative $\mathbb{F}_p$-algebras.

Now a purely quadratic $\kappa$-algebra $(A_\bullet,\epsilon)$ is determined by the product map $A_1\times A_1\to A_2$ and $\epsilon$ (see \S\ref{section on kappa-algebras}).
Therefore, we consider the category of \textsl{augmented $\mathbb{F}_p$-bilinear maps} $(b\colon V\times V\to W,\epsilon)$, where $V$ and $W$ are $\mathbb{F}_p$-linear spaces, and $\epsilon\in V$ is a distinguished element with the above properties (see \S\ref{section on augmented bilinear maps}).
Thus our question reduces to identifying within this category the augmented bilinear maps \textsl{of field type}, namely, those realizable as augmented bilinear maps $(\cup\colon H^1(F)\times H^1(F)\to H^2(F),(-1)_F)$
for some field $F$ containing a root of unity of order $p$.
We prove:

\begin{mainthm}
\label{main theorem}
The augmented $\mathbb{F}_p$-bilinear maps of field type are cofinal in the category of all augmented $\mathbb{F}_p$-bilinear maps.
That is, for every augmented $\mathbb{F}_p$-bilinear map $(b\colon V\times V\to W,\epsilon)$ there exists a field $F$ containing a root of unity of order $p$ and a monomorphism (in the natural sense)
\[
\bigl(b\colon V\times V\to W,\epsilon\bigr)\to \bigl(\cup\colon H^1(F)\times H^1(F)\to H^2(F),(-1)_F\bigr),
\]
where moreover, $W$ is mapped isomorphically onto $H^2(F)$.
\end{mainthm}   
See Theorem \ref{precise main theorem} for a precise statement and proof.

In Corollary \ref{main result for skew-symmetric bilinear maps} we prove an analogous result in the non-augmented context:
The cup product maps $\cup\colon H^1(F)\times H^1(F)\to H^2(F)$, for fields $F$ as above, are cofinal in the category of all skew-symmetric $\mathbb{F}_p$-bilinear maps $b\colon V\times V\to W$.

By the Norm Residue Theorem in degree 2 (which is due to Merkurjev and Suslin), augmented bilinear maps of field type are \textsl{surjective}, in the sense that the image of $b$ spans $W=H^2(F)$.
A natural question is whether in the Main Theorem one can take the monomorphism to be an isomorphism,
when the given augmented bilinear map is assumed to be surjective.
In \S\ref{section on the common slot property} we answer this question negatively, as follows:
For $p=2$, augmented $\mathbb{F}_2$-bilinear maps $(H^1(F),H^2(F),\cup,(-1)_F)$ of field type have an additional property -- the \textsl{`common slot property'} of the quaternionic pairing.
We provide examples of surjective augmented $\mathbb{F}_2$-bilinear maps which do not satisfy this property, and therefore are not of field type. 
In fact, it is a long-standing open problem in the algebraic theory of quadratic forms whether the above properties of $\epsilon$ together with the common slot property already completely characterize the $\mathbb{F}_2$-bilinear maps which are realizable as cup product maps over fields of characteristic $\neq2$ (see Remark \ref{quaternionic maps} as well as \cite{Marshall04}*{Open Problem 1 and Th.\ 2.2}, \cite{Lam05}*{Ch.\ XII, \S8}).

Our examples are constructed using \textsl{pro-$p$ right-angled Artin groups} $G_\Gamma$ associated with finite simplicial (unoriented) graphs $\Gamma$.
This means  $G_\Gamma$ is the pro-$p$ group with the vertices of $\Gamma$ as generators and with commutators of vertices connected by edges as defining relations.
The Galois-theoretic properties of the groups $G_\Gamma$ have been the object of extensive study in recent years.
In particular, Snopce and Zalesskii \cite{SnopceZalesskii22}, building on earlier results by Quadrelli \cite{Quadrelli14} and Cassella--Quadrelli \cite{CassellaQuadrelli21}, proved the following characterization theorem:
$G_\Gamma$ is realizable as the maximal pro-$p$ Galois group $G_F(p)$ of a field $F$ containing a root of unity of order $p$ if and only if $\Gamma$ does not contain a 4-vertex line $L_3$ or a 4-vertex circle $C_4$ as induced subgraphs.
See \cite{BlumerQuadrelliWeigel24} for a generalization to oriented graphs.

Now as we show in Theorem \ref{equivalence}, when $p=2$, the equivalent conditions in the Snopce--Zalesskii theorem are also equivalent to the realization of the augmented $\mathbb{F}_p$-bilinear map $(\cup\colon  H^1(G_\Gamma)\times H^1(G_\Gamma)\to H^2(G_\Gamma),0)$ by a field, as well as to the common slot property for this map. 
Our counterexamples are just the augmented $\mathbb{F}_2$-bilinear maps associated in this way with graphs $\Gamma$ which do not contain $L_3$ or $C_4$ as induced graphs. 

In \cite{SnopceZalesskii22} and \cite{Quadrelli14}, the groups $G_{L_3},G_{C_4}$ are ruled out from being maximal pro-$p$ Galois groups by showing that they are not \textsl{Bloch--Kato pro-$p$ groups}, that is, they have closed subgroups for which the mod-$p$ graded cohomology ring is not quadratic.
Our approach using the common slot property (for $p=2$) may therefore suggest a possible connection between the common slot property and the Bloch--Kato property.

We remark that there are additional \textsl{conjectural} properties of the graded algebras $H^\bullet(F)$ for arbitrary $p$, notably they are conjectured to be Koszul. 
See \cite{Positselski14}, \cite{MinacPasiniQuadrelliTan21}, \cite{MinacPasiniQuadrelliTan22}.

Our proof of the Main Theorem is based on Amitsur's construction of generic splitting fields of central simple algebras, or alternatively, elements of the second cohomology group $H^2(F)$ (See \S\ref{section on generic splitting fields}).
This construction is applied iteratively to produce fields whose cohomology groups realize the desired bilinear maps.
Transcendental constructions in this spirit were used by Min\'a\v c and Ware (\cite{MinacWare91}, \cite{MinacWare92}), as well as by Nikolov and the first-named author \cite{BarOnNikolov24}, in their works on the realization of pro-$p$ Demu\v skin groups of infinite rank as maximal pro-$p$ Galois groups of fields.

\medskip

We thank Uzi Vishne for a helpful discussion.
We also thank the referee for a careful reading of this manuscript and for their comments.

\section{Augmented bilinear maps}
\label{section on augmented bilinear maps}
Let $R$ be a commutative ring.
An \textsl{augmented $R$-bilinear map} $(V,W,b,\epsilon)$ consists of $R$-modules $V$ and $W$, an $R$-bilinear map 
\[
b\colon V\times V\to W,
\]
and a distinguished element $\epsilon\in V$, such that:
\begin{enumerate}
\item[(i)]
$b(v,\epsilon+v)=0$ for every $v\in V$;
\item[(ii)]
$2\epsilon=0$.
\end{enumerate}
We say that $(V,W,b,\epsilon)$ is \textsl{surjective} if the homomorphism $V\otimes_R V\to W$ induced by $b$ is surjective.

A \textsl{morphism} $f=(f_1,f_2)\colon(V,W,b,\epsilon)\to(V',W',b',\epsilon')$ of augmented $R$-bilinear maps consists of $R$-module homomorphisms $f_1\colon V\to V'$ and $f_2\colon W\to W'$ such that $f_1(\epsilon)=\epsilon'$, and the following diagram commutes:
\[
\xymatrix{
V\ar[d]_{f_1}&*-<3pc>{\times} &V\ar[d]^{f_1}\ar[r]^{b}&W\ar[d]^{f_2}\\
V'&*-<3pc>{\times}&V'\ar[r]^{b'}&W'.
}
\]
We call $f$ a \textsl{monomorphism} if both $f_1$ and $f_2$ are injective.
The augmented $R$-bilinear maps thus form a category, denoted {\bf Aug-}$R${\bf-Bilin}.

\begin{lem}
\label{basic properties of augmented cup products}
Let $(V,W,b,\epsilon)$ be an augmented $R$-bilinear map.
\begin{enumerate}
\item[(a)]
For every $v\in V$ one has $b(v,\epsilon)=b(v,v)=b(\epsilon,v)$. 
\item[(b)]
The map $b$ is \textsl{skew-symmetric}, that is, $b(v,v')=-b(v',v)$ for every $v,v'\in V$.
\item[(c)]
If $2$ is invertible in $R$, then $\epsilon=0$.
\end{enumerate}
\end{lem}
\begin{proof}
(a) \quad
The first equality follows from (i), (ii) and the bilinearity of $b$.
For the second equality, apply (i) with $v$ replaced by $\epsilon+v$.

\medskip

(b) \quad
For every $v,v'\in V$ we obtain from (i) and the bilinearity that
\[
\begin{split}
0=b(v+v',\epsilon+v+v')&=b(v,\epsilon+v)+b(v,v')+b(v',v)+b(v',\epsilon+v')\\
&=b(v,v')+b(v',v).
\end{split}
\]

\medskip
(c) is immediate from (ii).
\end{proof}

Subsequently, we will focus on the case where $R=\mathbb{F}_p$ for a prime $p$.
In this case we have the following partial converse of Lemma \ref{basic properties of augmented cup products}(b):

\begin{lem}
\label{bilinear extends to augmented bilinear}
Let $p$ be a prime number and let $b\colon V\times V\to W$ be a skew-symmetric $\mathbb{F}_p$-bilinear map.
Then there exists an augmented $\mathbb{F}_p$-bilinear map $(\widehat V,W,\widehat b,\epsilon)$ such that $V$ is an $\mathbb{F}_p$-linear subspace of $\widehat V$ of codimension at most $1$ and $\widehat b$ extends $b$.
\end{lem}
\begin{proof}
If $p\neq2$, then $b$ is alternate, i.e., $b(v,v)=0$ for every $v\in V$.
We may therefore take $\widehat V=V$, $\widehat b=b$, and $\epsilon=0$.

Suppose that $p=2$.
By the skew-symmetry of $b$, the map $V\to W$, $v\mapsto b(v,v)$, is $\mathbb{F}_2$-linear.
Let $\widehat V=V\oplus\Span(\epsilon)$ for a new vector $\epsilon$.
We define a map $\widehat b\colon\widehat V\times\widehat V\to W$ by 
\[
\widehat b(v+a\epsilon,v'+a'\epsilon)=b(v,v')+a'b(v,v)+ab(v',v')
\]
for $v,v'\in V$ and $a,a'\in \mathbb{F}_2$.
A direct computation shows that it is $\mathbb{F}_2$-bilinear.
We have 
\[
\widehat b(v+a\epsilon,v+a\epsilon)=b(v,v)=\widehat b(v+a\epsilon,\epsilon).
\]
Therefore $(\widehat V,W,\widehat b,\epsilon)$ is an augmented $\mathbb{F}_2$-bilinear map.
\end{proof}

\section{$\kappa$-Algebras}
\label{section on kappa-algebras}
Following Bass and Tate \cite{BassTate73}, we define a \textsl{$\kappa$-algebra} $(A_\bullet,\epsilon_A)$ over the ring $R$ to be a graded $R$-algebra $A_\bullet=\bigoplus_{i\geq0}A_i$ with a distinguished central element $\epsilon_A\in A_1$ such that
\begin{enumerate}
\item[(i)]
$a\cdot(\epsilon_A+a)=0$ for every $a\in A_1$;
\item[(ii)]
$2\epsilon_A=0$;
\item[(iii)]
$A_\bullet$ is generated as an $R$-algebra by $A_1$.
\end{enumerate}

A \textsl{morphism} $f\colon (A_\bullet,\epsilon_A)\to (B_\bullet,\epsilon_B)$ of $\kappa$-algebras over $R$ is a graded $R$-algebra morphism $f\colon A_\bullet\to B_\bullet$ such that $f(\epsilon_A)=\epsilon_B$.
The $\kappa$-algebras over $R$ thus form a category, denoted $\kappa$-{\bf Alg}$_R$.

Given an $R$-module $V$, let ${\rm T}_\bullet(V)$ denote the tensor graded $R$-algebra over $V$.
For a $\kappa$-algebra $(A_\bullet,\epsilon_A)$, there is a unique morphism $f\colon ({\rm T}_\bullet(A_1),\epsilon_A)\to (A_\bullet,\epsilon_A)$ of $\kappa$-algebras which is the identity on $A_1$.
One says that $A_\bullet$ is \textsl{quadratic} (resp., \textsl{purely quadratic}) if $f$ is surjective and  $\Ker(f)$ is generated as a two-sided ideal of ${\rm T}_\bullet(A_1)$ by elements of $A_1\otimes_R A_1$ (resp., pure tensors $a\otimes a'$ in $A_1\otimes_R A_1$).
 
We define a functor
\[
{\bf F}\colon\kappa\hbox{\bf-Alg}_R
\to 
\hbox{\bf Aug-}R\hbox{\bf-Bilin}  
\]
by mapping a $\kappa$-algebra $(A_\bullet=\bigoplus_{i\geq0}A_i,\epsilon_A)$ over $R$ to the augmented $R$-bilinear map $(A_1,A_2,\mu,\epsilon_A)$, where $\mu\colon A_1\times A_1\to A_2$ is the multiplication map in $A_\bullet$.

\begin{rem}
\rm
By Lemma \ref{basic properties of augmented cup products}(b), $\mu$ is skew-symmetric.
In view of (iii), this implies that $A_\bullet$ is graded-commutative (see \cite{BassTate73}*{Prop.\ 2.2(a)}).
\end{rem}

Conversely, we define a functor
\[
{\bf G}\colon \hbox{\bf Aug-}R\hbox{\bf-Bilin} \to \kappa\hbox{\bf-Alg}_R
\]
by mapping an augmented $R$-bilinear map $(V,W,b,\epsilon)$ to the purely quadratic $\kappa$-structure $({\rm T}_\bullet(V)/I,\epsilon)$, where $I$ is the two-sided (graded) ideal in ${\rm T}_\bullet(V)$ generated by all pure tensors $v\otimes v'$, with $v,v'\in V$ satisfying $b(v,v')=0$.

The composition ${\bf G}\circ{\bf F}$ is the \textsl{purely quadratic hull} functor; see \cite{CheboluEfratMinac12}*{\S3}.

\section{Galois cohomology}
\label{section on Galois cohomology}
We now fix a prime number $p$ and take the ring $R=\mathbb{F}_p$.
Given a profinite group $G$, we write $H^i(G)=H^i(G,\mathbb{F}_p)$ for the $i$th profinite cohomology group of $G$ with respect to its trivial action on $\mathbb{F}_p$.
Then the cup product
\[
\cup\colon H^i(G)\times H^j(G)\to H^{i+j}(G)
\]
is a skew-symmetric $\mathbb{F}_p$-bilinear map \cite{NeukirchSchmidtWingberg}*{Ch.\ I, Prop.\ 1.4.4}.

As in the Introduction, we consider the cup product in a Galois-theoretic context as follows:
Let $F$ be a field and let $G_F={\rm Gal}(F^{\rm sep}/F)$ be its absolute Galois group, where $F^{\rm sep}$ denotes the separable closure of $F$.
We will further assume that $F$ contains a root of unity of order $p$, so in particular, ${\rm Char}\,F\neq p$. 
We fix an isomorphism between the group $\mu_p$ of all roots of unity of order $p$ in $F^{\rm sep}$ and $\mathbb{Z}/p$.
Let $\zeta\in\mu_p$ correspond to $\bar1\in\mathbb{Z}/p$.
For simplicity we abbreviate $H^i(F)=H^i(G_F)$.
We write again $H^\bullet(F)=\bigoplus_{i\geq0}H^i(F)$ for the cohomology graded $\mathbb{F}_p$-algebra with the cup product.

Given a field extension $E/F$, we write $\Res_{E/F}=\Res^i_{E/F}\colon H^i(F)\to H^i(E)$ for the functorial (restriction) map.
Note that here we do not assume that $E/F$ is algebraic.

One has the Kummer isomorphism 
\[
F^\times/(F^\times)^p\isom H^1(G_F,\mu_p)=H^1(F).
\]
To every $a\in F^\times$ we associate its Kummer element $(a)_F\in H^1(F)$.

Further, the identification $\mathbb{Z}/p=\mu_p$ induces a canonical isomorphism between $H^2(F)$ and the $p$-torsion part ${}_p\mathrm{Br}(F)$ of the Brauer group of $F$.
Under this isomorphism, a cup product $(a_1)_F\cup(a_2)_F$, where $a_1,a_2\in F^\times$, corresponds to the similarity class of the \textsl{symbol algebra} 
$(a_1,a_2)_{F,\zeta}$ over $F$, generated by elements $t_1,t_2$ subject to the relations $t_1^p=a_1$, $t_2^p=a_2$, and $t_1t_2=\zeta t_2t_1$.

Recall that the \textsl{Milnor ${\rm K}$-algebra} of the field $F$ is the graded $\mathbb{Z}$-algebra 
\[
{\rm K}^M_\bullet(F)={\rm T}_\bullet(F^\times)/J,
\] 
where ${\rm T}_\bullet(F^\times)$ is the tensor $\mathbb{Z}$-algebra over the multiplicative group $F^\times$ of $F$, and $J$ is its two-sided (graded) ideal generated by all pure tensors $a\otimes (1-a)$, with $0,1\neq a\in F$.
It is a purely quadratic $\kappa$-algebra over $\mathbb{Z}$ with the distinguished element $\epsilon_K=-1$ \cite{BassTate73}*{Example (2.1)}.
Therefore ${\rm K}^M_\bullet(F)\otimes(\mathbb{Z}/p)$ is a purely quadratic $\kappa$-algebra over $\mathbb{F}_p$.

By the Norm Residue Theorem (\cite{HaesemeyerWeibel19}, \cite{Voevodsky03}, \cite{Voevodsky11}), the Kummer map $F^\times\to H^1(F)$, $a\mapsto (a)_F$, induces a graded algebra isomorphism
\begin{equation}
\label{norm residue isomorphism}
{\rm K}^M_\bullet(F)\otimes(\mathbb{Z}/p)\xrightarrow{\sim}H^\bullet(F).
\end{equation}
It follows that $(H^\bullet(F),(-1)_F)$ is a $\kappa$-algebra over $\mathbb{F}_p$.

By applying the functor $\bf F$, we deduce that the cup product map $\cup\colon H^1(F)\times H^1(F)\to H^2(F)$ gives rise to a surjective augmented $\mathbb{F}_p$-bilinear map
\begin{equation}
\label{augmented bilinear maps of field type}
\bigl(H^1(F),H^2(F),\cup,(-1)_F\bigr).
\end{equation}

Given a field extension $E/F$, we have a morphism of $\kappa$-structures
\[
\Res=\oplus_{i\geq0}\Res_{E/F}^i\colon H^\bullet(F)\to H^\bullet(E),
\]
as well as a morphism of augmented $\mathbb{F}_p$-bilinear maps
\[
\begin{split}
\Res_{E/F}&=(\Res_{E/F}^1,\Res_{E/F}^2)\colon\\ 
&\bigl(H^1(F),H^2(F),\cup,(-1)_F\bigr)\to \bigl(H^1(E),H^2(E),\cup,(-1)_E\bigr).
\end{split}
\]

We define $\kappa$-{\bf Alg-Fields} to be the subcategory of $\kappa$-{\bf Alg}$_{\mathbb{F}_p}$ consisting of all $\kappa$-algebras over $\mathbb{F}_p$ of the form $(H^\bullet(F),(-1)_F)$, and with restriction maps as morphisms.
We call its objects \textsl{$\kappa$-algebras over $\mathbb{F}_p$ of field type}.

Likewise we define  {\bf Aug-Bilin-Fields} to be the subcategory of {\bf Aug-}$\mathbb{F}_p${\bf-Bilin} consisting of all (surjective) augmented $\mathbb{F}_p$-bilinear maps of the form (\ref{augmented bilinear maps of field type}) for some $F$, with restriction maps as morphisms.
We call its objects \textsl{augmented $\mathbb{F}_p$-bilinear maps of field type}.

\begin{thm}
\label{equivalence of the Galois categories}
The restricted functors
\[
\hbox{$\kappa$-{\bf Alg-Fields}} 
\ \  
\begin{array}{c}
\overset{\bf F}{\xrightarrow{\qquad\quad}} \\
\underset{\bf G}{\xleftarrow{\qquad\quad}}
\end{array}
\ \ 
\hbox{{\bf Aug-Bilin-Fields}}  
\]
form an equivalence of categories.
\end{thm}
\begin{proof}
It is straightforward to verify that $\bf F$ and $\bf G$ map restriction morphisms to restriction morphisms.

Next take a field $F$ containing a root of unity of order $p$.
We denote the image under ${\bf G}\circ{\bf F}$ of the $\kappa$-algebra $(H^\bullet(F),(-1)_F)$ by
\[
(\widehat H^\bullet(F)=\bigoplus_{i\geq0}\widehat H^i(F),(-1)_F).
\]
Thus $\widehat H^\bullet(F)={\rm T}_\bullet(H^1(F))/I$, where $I$ is the two-sided ideal generated by all pure tensors $(a)_F\otimes(b)_F$ with $(a)_F\cup(b)_F=0$ in $H^2(F)$.
Since $(a)_F\cup(1-a)_F=0$ for every $0,1\neq a\in F$, the isomorphism (\ref{norm residue isomorphism}) factors as
\[
{\rm K}^M_\bullet(F)\otimes(\mathbb{Z}/p)\to \hat H^\bullet(F)\to H^\bullet(F),
\]
In view of the Kummer isomorphism, the left map is surjective, and we deduce that both maps are isomorphisms.
This implies that ${\bf G}\circ{\bf F}={\bf Id}$ on $\kappa$-{\bf Alg-Fields}.

For the second composition, we observe that
\begin{equation}
\label{F circ G}
({\bf F}\circ{\bf G})(H^1(F),H^2(F),\cup,(-1)_F)=(\widehat H^1(F),\widehat H^2(F),b,(-1)_F),
\end{equation}
where $b\colon \widehat H^1(F)\times \widehat H^1(F)\to \widehat H^2(F)$ is induced by the tensor product.
Now $\widehat H^1(F)=H^1(F)$.
Further, for the degree $2$ component $I_2$ of $I$, we obtain from the previous paragraph that the cup product induces an isomorphism $\widehat H^2(F)=H^1(F)^{\otimes2}/I_2\to H^2(F)$. 
Hence the right hand-side of (\ref{F circ G}) is isomorphic to $(H^1(F),H^2(F),\cup,(-1)_F)$.
Therefore ${\bf F}\circ{\bf G}$ is naturally isomorphic to the identity functor $\bf Id$ on {\bf Aug-Bilin-Fields}.
\end{proof}

\begin{rem}
\rm
In the above discussion, we may replace the absolute Galois group $G_F$ by its maximal pro-$p$ quotient $G_F(p)$.
Indeed, it is a consequence of the Norm Residue Theorem that the inflation map $\inf\colon H^i(G_F(p))\to H^i(G_F)$ is an isomorphism for all $i$ \cite{CheboluEfratMinac12}*{Remark 8.2}.
\end{rem}

\section{Generic splitting fields}
\label{section on generic splitting fields}
The following proposition restates in cohomological terms some main facts about Amitsur's construction of generic splitting fields of central simple algebras (\cite{Amitsur55}, \cite{Amitsur82}, \cite{Roquette63}):

\begin{prop}
\label{generic splitting fields}
Given a field $K$ and a cohomology element $\omega\in H^2(K)$, there exists a regular extension $L$ of $K$ with finite relative transcendence degree such that the following sequence is exact:
\[
0\to\langle\omega\rangle\to H^2(K)\xrightarrow{\Res_{L/K}}H^2(L).
\]
\end{prop}
By a slight abuse of notation, we call $L$ a \textsl{generic splitting field} of $\omega$ over $K$.

We fix an algebraically closed field $\Omega$ of sufficiently large transcendence degree over the ground field.
We may assume that all generic splitting fields in the following constructions will be taken inside $\Omega$.

\begin{prop}
\label{killing all algebras}
Let $K$ be a field containing a root of unity of order $p$ and let $A$ be a subgroup of $H^2(K)$. 
Then there is a field extension $L$ of $K$ such that
\begin{enumerate}
\item[(i)]
$\Res_{L/K}\colon H^1(K)\to H^1(L)$ is injective;  and
\item[(ii)]
$\Res_{L/K}\colon A\to H^2(L)$ is an isomorphism.  
\end{enumerate}
\end{prop}
\begin{proof}
We consider the set $\mathcal{S}$ of all intermediate fields $K\subseteq L\subseteq\Omega$ such that (i) holds as well as:
\\
\null\quad  \ (ii') \   $\Res_{L/K}\colon H^2(K)\to H^2(L)$ is injective on $A$.
\\
It clearly contains $K$.

We partially order $\mathcal{S}$ by setting $L\preceq L'$ if and only if $L\subseteq L'$, the map $\Res_{L'/L}\colon H^1(L)\to H^1(L')$ is injective, and $\Res_{L'/L}\colon H^2(L)\to H^2(L')$ is injective on $\Res_{L/K}(A)$.

For a nonempty totally ordered subset $\mathcal{C}$ of $\mathcal{S}$, the union $L_\mathcal{C}$ of the fields in $\mathcal{C}$ satisfies $H^i(L_\mathcal{C})=\varinjlim_{L\in\mathcal{C}}H^i(L)$ for every $i$.
Hence $L_\mathcal{C}\in\mathcal{S}$, and $L\preceq L_\mathcal{C}$ for every $L\in\mathcal{C}$.
Zorn's lemma therefore yields a maximal field $L$ in $(\mathcal{S},\preceq)$.

It remains to prove that $\Res_{L/K}(A)=H^2(L)$.
To this end suppose that $\omega\in H^2(L)$ is not in $\Res_{L/K}(A)$.
Let $L'$ be a generic splitting field of $\omega$ over $L$.
It is a regular extension of $L$, so $\Res_{L'/L}\colon H^1(L)\to H^1(L')$ is injective.
The kernel $\langle\omega\rangle$ of $\Res_{L'/L}\colon H^2(L)\to H^2(L')$ has trivial intersection with $\Res_{L/K}(A)$, implying that $\Res_{L'/L}\colon H^2(L)\to H^2(L')$ is injective on $\Res_{L/K}(A)$.
It follows using (ii') that $\Res_{L'/K}\colon H^2(K)\to H^2(L')$ is injective on $A$.
We conclude that $L'\in\mathcal{S}$.
The maximality of $L$ in $\mathcal{S}$ implies that $L'=L$, contrary to $\omega\neq0$. 
\end{proof}

\section{Power series fields}
Let $\Gamma=(\Gamma,\leq)$ be an ordered abelian group.
Let $K$ be a field and let $E=K((\Gamma))$ be the field of all formal power series $z=\sum_{\gamma\in \Gamma}a_\gamma t^\gamma$ with coefficients $a_\gamma$ in $K$, exponents in $\Gamma$, and well ordered support $\Supp(z)=\{\gamma\in\Gamma\ |\  a_\gamma\neq0\}$ \cite{Efrat06}*{\S2.8}.
It is Henselian with respect to the valuation $v\colon E^\times\to\Gamma$ defined by $v(z)=\min(\Supp(z))$ \cite{Efrat06}*{Cor.\ 18.4}.
We identify the elements of $K$ as constant power series in $E$, thus making $K$ a subfield of $E$.
For instance, $K((\mathbb{Z}))$, with the canonical order of $\mathbb{Z}$, is the usual formal power series field $K((t))$ with its canonical valuation.

The cohomology groups of $K$ and $E$ are related as follows \cite{Wadsworth83}*{\S3}:
We choose a totally ordered set $(\Lambda,\leq)$ and elements $\gamma_l$, $l\in \Lambda$, of $\Gamma$ whose cosets form an $\mathbb{F}_p$-linear basis of $\Gamma/p\Gamma$.
Then for every $i\geq0$,
\[
H^i(E)=\bigoplus_{j=0}^i\bigoplus_{{l_1,\ldots, l_j\in \Lambda}\atop{l_1<\cdots<l_j}}\Res_{E/K}(H^{i-j}(K))\cup (t^{l_1})_E\cup\cdots\cup(t^{l_j})_E.
\]
Moreover, for every $0\leq j\leq i$ and $l_1<\cdots<l_j$ in $\Lambda$, the map 
\[
H^{i-j}(K)\to H^i(E), \quad
\omega\mapsto \Res_{E/K}(\omega)\cup (t^{l_1})_E\cup\cdots\cup(t^{l_j})_E, 
\]
is injective.
In particular,
\begin{equation}
\label{structure of cohomology in extensions}
\begin{split}
H^1(E)&\isom H^1(K)\oplus(\Gamma/p\Gamma),\\
H^2(E)&\isom  H^2(K)\oplus\bigl(H^1(K)\otimes(\bigoplus_{l\in \Lambda}(\mathbb{Z}/p)\bigr)\oplus\bigoplus_{l_1<l_2}(\mathbb{Z}/p).
\end{split}
\end{equation}

We immediately obtain:

\begin{lem}
\label{injectivity of restriction in power field  extensions}
For $E=K((\Gamma))$, the restriction morphism
\[
\Res_{E/K}\colon (H^1(K),H^2(K),\cup,(-1)_K)\to(H^1(E),H^2(E),\cup,(-1)_E)
\]
is a monomorphism of augmented $\mathbb{F}_p$-bilinear maps.
\end{lem}

\begin{lem}
\label{H2 large}
Denoting $m=\dim_{\mathbb{F}_p}(\Gamma/p\Gamma)$, one has $\dim_{\mathbb{F}_p}H^2(E)\geq m$ unless $H^1(K)=0$ and $m\leq2$.
\end{lem}
\begin{proof}
This follows from (\ref{structure of cohomology in extensions}).
\end{proof}

We now prove a weak form of the Main Theorem in the special case where $V=\Span(\epsilon)$.

\begin{prop}
\label{the case V=epsilon}
Let $(V,W,b,\epsilon)$ be an augmented $\mathbb{F}_p$-bilinear map with $V=\Span(\epsilon)$.
Then there exists a field $E$ containing a root of unity of order $p$ and a monomorphism 
\[
(f_1,f_2)\colon (V,W,b,\epsilon)\to(H^1(E),H^2(E),\cup,(-1)_E).
\]
\end{prop}
\begin{proof}
First we take a field $K$ as follows:
\begin{itemize}
\item
If $\epsilon=0$, let $K=\mathbb{C}$. 
Here $(-1)_K=0$.
\item
If $\epsilon\neq0$ (so $p=2$) and $b(\epsilon,\epsilon)=0$, let $K$ be a field such that $-1$ is a sum of two squares but is not a square (for instance, an algebraic extension of $\mathbb{Q}$ with $G_K\isom\mathbb{Z}_2$ and $\sqrt{-1}\not\in K$, which exists e.g., by \cite{FriedJarden23}*{Th.\ 21.5.6}).
Here $(-1)_K\neq0$, but $(-1)_K\cup(-1)_K=0$.
\item
If $b(\epsilon,\epsilon)\neq0$ (so $\epsilon\neq0$ and $p=2$), let $K=\mathbb{Q}_2$.
Here $-1$ is not a sum of two squares, so $(-1)_K\cup(-1)_K\neq0$ (See \cite{Lam05}*{Ch.\ VI, Cor.\ 2.24(4)}).
\end{itemize} 
Note that in all cases, $K$ contains a root of unity of order $p$.

Next let $m=\max\{\dim_{\mathbb{F}_p}(W),3\}$.
Let $\Gamma=\mathbb{Z}^m$, considered as an ordered abelian group with respect to the lexicographic order induced by the natural order on $\mathbb{Z}$, and set $E=K((\Gamma))$.
Then $(-1)_E=0$ if $\epsilon=0$, and $(-1)_E\cup(-1)_E=0$ if $b(\epsilon,\epsilon)=0$.
By Lemma \ref{H2 large}, $\dim_{\mathbb{F}_p}(H^2(E))\geq m$.

We now define an $\mathbb{F}_p$-linear map $f_1\colon V=\Span(\epsilon)\to H^1(E)$ by $f_1(\epsilon)=(-1)_E$.
We further choose an $\mathbb{F}_p$-linear monomorphism $f_2\colon W\to H^2(E)$ satisfying \[
f_2(b(\epsilon,\epsilon))=(-1)_E\cup(-1)_E.
\]
Then $(f_1,f_2)$ is a monomorphism of augmented $\mathbb{F}_p$-bilinear maps, as desired.
\end{proof}

\section{A transfinite construction}
The following proposition will be the key step in the proof of the Main Theorem.  

\begin{prop}
\label{extending embeddings in codimension 1}
Let $(U,W,b,\epsilon)$ be an augmented $\mathbb{F}_p$-bilinear map.
Suppose that $\widehat U$ is an $\mathbb{F}_p$-linear space containing $U$ as a subspace of codimension $1$, and let $(\widehat U, W,\widehat b,\epsilon)$ be an augmented $\mathbb{F}_p$-bilinear map with $\widehat b|_{U\times U}=b$ (but with the same space $W$).
Let $K$ be a field containing a root of unity of order $p$, and let $(f_1,f_2)$ be a monomorphism as in the lower row of the diagram below. 

Then there exists a field extension $\widehat K/K$, with $\Res_{\widehat K/K}\colon H^1(K)\to H^1(\widehat K)$ injective,
and a monomorphism $(\widehat f_1,\widehat f_2)$ of augmented $\mathbb{F}_p$-bilinear maps, making the following diagram commutative: 
\[
\xymatrix{
\ (\widehat U,W,\widehat b,\epsilon)\ \ar@{-->}@{>->}[r]^{(\widehat f_1,\widehat f_2)\quad\qquad}&\ (H^1(\widehat K),H^2(\widehat K),\cup,(-1)_{\widehat K})\ \\
\ \strut (U,W,b,\epsilon)\ \ar@{^{(}->}[u]\ar@{>->}[r]^{(f_1,f_2)\quad\qquad}&\ \strut (H^1(K),H^2(K),\cup,(-1)_K)\ \ar[u]_{\Res_{\widehat K/K}}.
}
\]
\end{prop}
\begin{proof}
By (\ref{structure of cohomology in extensions}), one has
\[
H^1(K((t)))=\Res_{K((t))/K}(H^1(K))\oplus \langle\delta\rangle,
\]
where $\delta=(t)_{K((t))}$.
We may therefore replace $K$ by $K((t))$ to assume, using Lemma \ref{injectivity of restriction in power field  extensions}, that there exists $0\neq\delta\in H^1(K)$ such that 
\begin{equation}
\label{cohomology of the extension}
f_1(U)\cap\langle\delta\rangle=\{0\}.
\end{equation}

Next let $v\in \widehat U\setminus U$, so $\widehat U=U\oplus\Span(v)$.
We further choose a well ordered basis $\{u_{\alpha}\}_{1\leq \alpha<\mu}$ for $U$, where $\mu$ is an ordinal. 
We write 
\[
V_\alpha=\Span\bigl\{u_\beta\ \bigm|\ \beta\leq\alpha\bigr\}.
\]

We construct by transfinite induction a chain of fields $L_\alpha$, $\alpha<\mu$, with the following properties:
\begin{enumerate}
\item[(i)] 
For all $\beta\leq\alpha$, $L_\beta\subseteq L_\alpha$.
\item[(ii)] 
For all $\beta\leq\alpha$, the map $\Res_{L_\alpha/L_\beta}\colon H^1(L_\beta)\to H^1(L_\alpha)$ is injective.
\item[(iii)] 
For all $1\leq\alpha<\mu$ and $u\in V_\alpha$, 
\[
\Res_{L_\alpha/K}(f_1(u)\cup\delta)=\Res_{L_\alpha/K}(f_2(\widehat b(u,v))).
\]
\item[(iv)] 
For all $1\leq\alpha<\mu$ and $\chi\in H^1(K)$,  if 
\[
\Res_{L_\alpha/K}(\chi\cup\delta)\in\Res_{L_\alpha/K}(\Img(f_2)),
\]
then $\chi\in f_1(V_\alpha)$.
\item[(v)] 
For all $1\leq\alpha<\mu$, the restriction of $\Res_{L_\alpha/K}$ to $\Img(f_2)$ is injective.
\end{enumerate}

\bigskip

Case I: $\alpha=0$.
We take $L_0=K$, and there is nothing to prove.

\bigskip

Case II:\  $\alpha=\gamma+1$ for some $\gamma$. 
Let $L_\alpha$ be a generic splitting field of 
\[
\Res_{L_\gamma/K}(f_1(u_\alpha)\cup\delta-f_2(\widehat b(u_\alpha,v)))\in H^2(L_\gamma).
\] 

Property (i) is immediate.

For Property (ii) note that since a generic splitting field is a regular extension, $\Res_{L_\alpha/L_\gamma}\colon H^1(L_\gamma)\to H^1(L_\alpha)$ is injective.
Further, by induction, $\Res_{L_\gamma/L_\beta}\colon H^1(L_\beta)\to H^1(L_\gamma)$ is injective for all $\beta\leq\gamma$.
 
Property (iii) holds for $u=u_\alpha$ by the choice of $L_\alpha$.
Since, by induction, it is valid for every $u\in V_\gamma$, we deduce that it holds for all $u\in V_\alpha=V_\gamma\oplus\Span(u_\alpha)$.

To show Property (iv), we take $\chi\in H^1(K)$ and $\omega\in\Img(f_2)$ such that 
$\Res_{L_\alpha/K}(\chi\cup\delta)=\Res_{L_\alpha/K}(\omega)$.
By the properties of a generic splitting field (Proposition \ref{generic splitting fields}), there exists an integer $m$ such that 
\[
\Res_{L_\gamma/K}(\chi\cup\delta-\omega)
=m\Res_{L_\gamma/K}(f_1(u_\alpha)\cup\delta-f_2(\widehat b(u_{\alpha},v))).
\]
It follows that
\[
\Res_{L_\gamma/K}((\chi-mf_1(u_\alpha))\cup\delta)\in \Res_{L_\gamma/K}(\Img(f_2)).
\] 
By Property (iv) for $L_\gamma$, this implies that 
$\chi-mf_1(u_\alpha)\in f_1(V_\gamma)$, 
and therefore $\chi\in f_1(V_\alpha)$, as desired.

To prove Property (v), suppose that $\Res_{L_\alpha/K}(\omega)=0$ for some $0\neq\omega\in \Img(f_2)$.
Thus $\Res_{L_\gamma/K}(\omega)\in\Ker(\Res_{L_\alpha/L_\gamma})$.
The exact sequence in Proposition \ref{generic splitting fields} yields $0\leq m<p$ such that
\[
\Res_{L_\gamma/K}(\omega)
=m\Res_{L_\gamma/K}(f_1(u_\alpha)\cup\delta-f_2(\widehat b(u_\alpha,v))).
\]
By induction, $\Res_{L_\gamma/K}(\omega)\neq0$, so necessarily $m\neq 0$.
It follows that
\[
\Res_{L_\gamma/K}(f_1(u_\alpha)\cup\delta)\in \Res_{L_\gamma/K}(\Img(f_2)).
\]
By Property (iv) for $L_\gamma$, this implies that $f_1(u_\alpha)\in f_1(V_\gamma)$.
Since $f_1$ is injective, $u_\alpha\in V_\gamma$,  contrary to the linear independence of $u_\beta$, $\beta<\mu$.

\bigskip
	
Case III: $\alpha$ is a limit ordinal.
Let $L_\alpha=\bigcup_{\gamma<\alpha}L_\gamma$.
Then $G_{L_\alpha}=\varprojlim_{\gamma<\alpha} G_{L_\gamma}$, so $H^i(L_\alpha)=\varinjlim_{\gamma<\alpha}H^i(L_\gamma)$, $i=1,2$.
Properties (i)--(v) now follow by a standard limit argument.

\bigskip

To complete the proof, we take $\widehat K=\bigcup_{\alpha<\mu}L_\alpha$.
It follows from Property (ii) that 
$\Res_{\widehat K/K}\colon H^1(K)\to H^1(\widehat K)=\varinjlim_{\alpha<\mu}H^1(L_\alpha)$ is injective.

Take $\alpha$ such that $\epsilon\in V_\alpha$.
Property (iii) for $L_\alpha$ implies that 
\[
\Res_{\widehat K/K}(f_1(\epsilon)\cup\delta)=\Res_{\widehat K/K}(f_2(\widehat b(\epsilon,v))).
\]
Since $(-1)_K=f_1(\epsilon)$,  we obtain using Lemma \ref{basic properties of augmented cup products}(b) that
\begin{equation}
\label{sss}
\begin{split}
\Res_{\widehat K/K}(\delta\cup\delta)&=\Res_{\widehat K/K}((-1)_K\cup\delta)
=\Res_{\widehat K/K}(f_1(\epsilon)\cup\delta)\\
&=\Res_{\widehat K/K}(f_2(\widehat b(\epsilon,v)))
=\Res_{\widehat K/K}(f_2(\widehat b(v,v)).
\end{split}
\end{equation}

We now define $\mathbb{F}_p$-linear maps 
\[
\widehat f_1\colon\widehat U=U\oplus\Span(v)\to H^1(\widehat K), \quad \widehat f_2\colon W\to H^2(\widehat K)
\]
by
\[
\widehat f_1|_U=\Res_{\widehat K/K}\circ f_1, \quad \widehat f_1(v)=\Res_{\widehat K/K}(\delta), \quad \widehat f_2=\Res_{\widehat K/K}\circ f_2.
\]
We show that $(\widehat f_1,\widehat f_2)$ is a morphism of augmented $\mathbb{F}_p$-bilinear maps.
Indeed, 
\[
\widehat f_1(\epsilon)=\Res_{\widehat K/K}(f_1(\epsilon))=\Res_{\widehat K/K}((-1)_K)=(-1)_{\widehat K}.
\]
Furthermore, take $\hat u,\hat u'\in \widehat U$, and write $\hat u=u+mv$, $\hat u'=u'+m'v$ with $u,u'\in U$ and $m,m'\in\mathbb{Z}$.
Using Property (iii) and equation (\ref{sss}), we compute:
\[
\begin{split}
&\widehat f_1(\hat u)\cup\widehat f_1(\hat u')
=\Res_{\widehat K/K}\bigl((f_1(u)+m\delta)\cup(f_1(u')+m'\delta)\bigr) \\
&=\Res_{\widehat K/K}\bigl(f_1(u)\cup f_1(u')+m'f_1(u)\cup\delta-mf_1(u')\cup\delta+mm'\delta\cup\delta\bigr) \\
&=\Res_{\widehat K/K}\bigl(f_2(b(u,u'))+m'f_2(\widehat b(u,v))-mf_2(\widehat b(u',v))+mm'f_2(\widehat b(v,v))\bigr) \\
&=\Res_{\widehat K/K}\bigl(f_2(\widehat b(u,u')+m'\widehat b(u,v)-m\widehat b(u',v)+mm'\widehat b(v,v)\bigr) \\
&=\Res_{\widehat K/K}\bigl(f_2(\widehat b(\hat u,\hat u'))\bigr) \\
&=\widehat f_2(\widehat b(\hat u,\hat u')),
\end{split}
\]
as desired.

The commutativity of the diagram in the assertion of the proposition follows immediately from the definition of $\widehat f_1$ and $\widehat f_2$.

Finally, by Property (ii), $\Res_{\widehat K/K}$ is injective on $H^1(K)$.
Therefore the injectivity of $\widehat f_1$ follows from (\ref{cohomology of the extension}). 
Furthermore, by Property (v), $\Res_{\widehat K/K}$ is injective on $\Img(f_2)$, implying the injectivity of $\widehat f_2$. 
\end{proof}

\section{The main result}
Given augmented $\mathbb{F}_p$-bilinear pairs $(V,W,b,\epsilon)$ and $(V',W',b',\epsilon')$, we set 
$(V,W,b,\epsilon)\leq(V',W',b',\epsilon')$ if and only if there exists a monomorphism $(f_1,f_2)\colon (V,W,b,\epsilon)\to(V',W',b',\epsilon')$.
Thus our Main Theorem asserts that the subcategory {\bf Aug-Bilin-Fields} is cofinal in {\bf Aug-}$\mathbb{F}_p${\bf-Bilin} with respect to $\leq$ in the following (slightly stronger) sense:

\begin{thm}
\label{precise main theorem}
Let $(V,W,b,\epsilon)$ be an augmented $\mathbb{F}_p$-bilinear map. 
Then there exists a field $F$ containing a root of unity of order $p$, and a monomorphism 
\[
(f_1,f_2) \colon(V,W,b,\epsilon)\to (H^1(F),H^2(F),\cup,(-1)_F)
\]
of augmented $\mathbb{F}_p$-bilinear maps such that $f_2$ is an isomorphism.
\end{thm}	
\begin{proof}
Let $(V,W,b,\epsilon)$ be an augmented $\mathbb{F}_p$-bilinear map.
Let $\mathcal{S}$ be the set of all triples $(U,E,g)$ consisting of:
\begin{enumerate}
\item
A subspace $U$ of $V$ containing $\epsilon$;
\item
A field $E$ containing a root of unity of order $p$;
\item
A monomorphism of augmented $\mathbb{F}_p$-bilinear maps
\[
g=(g_1,g_2)\colon(U,W,b|_{U\times U},\epsilon)\to (H^1(E),H^2(E),\cup,(-1)_E).
\]
\end{enumerate}
Proposition \ref{the case V=epsilon} shows that $\mathcal{S}$ is nonempty, since $U=\Span\{\epsilon\}$ yields a triple in $\mathcal{S}$.

We partially order $\mathcal{S}$ by setting $(U,E,g)\preceq (U',E',g')$ if and only if $U\leq U'$, $E$ is a subfield of $E'$ such that $\Res_{E'/E}\colon H^1(E)\to H^1(E')$ is injective, and there is a commutative diagram of morphisms of augmented $\mathbb{F}_p$-bilinear maps:
\[
\xymatrixcolsep{5pc}
\xymatrix{
\ ( U',W,b,\epsilon)\ \ar@{-->}@{>->}[r]^(0.45){g'=(g'_1, g'_2)\qquad}&\ (H^1(E'),H^2(E'),\cup,(-1)_{E'})\ \\
	\ \strut (U,W,b,\epsilon)\ \ar@{^{(}->}[u]\ar@{>->}[r]^{g=(g_1,g_2)\qquad}&\ \strut (H^1(E),H^2(E),\cup,(-1)_E)\ \ar[u]_{\Res_{E'/E}}.
}
\]

Every chain $\mathcal{C}=\{(U_i,E_i,g_i=(g_{i,1},g_{i,2}))\ |\ i\in I\}$ in $(\mathcal{S},\preceq)$ admits an upper bound $(U,E,g)$ in $\mathcal{S}$, where 
\[
U=\bigcup_{i\in I} U_i, \quad E=\bigcup_{i\in I}E_i, \quad g=(\varinjlim_{i\in I} g_{i,1},\varinjlim_{i\in I}g_{i,2}).
\] 
Zorn's lemma therefore yields a maximal element $(U,K,f=(f_1,f_2))$ in $(\mathcal{S},\preceq)$.
 
If $U\neq V$, then we choose an intermediate subspace $U\subset \widehat U\subseteq V$, where $U$ has codimension $1$ in $\widehat U$.
Proposition \ref{extending embeddings in codimension 1} yields a field extension $\widehat K/K$ such that $\Res_{\widehat K/K}\colon H^1(K)\to H^1(\widehat K)$ is injective, and a monomorphism $\widehat f=(\widehat f_1,\widehat f_2)$ of augmented $\mathbb{F}_p$-bilinear maps with a commutative diagram:
\[
\xymatrixcolsep{5pc}\xymatrix{
\ (\widehat U,W,b|_{\widehat U\times \widehat U},\epsilon)\ \ar@{-->}@{>->}[r]^(0.45){( \widehat f_1,\widehat f_2)\qquad}&\ (H^1(\widehat K),H^2(\widehat K),\cup,(-1)_{\widehat K})\ \\
	\ \strut (U,W,b|_{U\times U},\epsilon)\ \ar@{^{(}->}[u]\ar@{>->}[r]^{(f_1,f_2)\qquad}&\ \strut (H^1(K),H^2(K),\cup,(-1)_K)\ \ar[u]_{\Res_{\widehat K/K}}.
}
\]
This contradicts the maximality of $(U,K,f)$.
Consequently, $U=V$.

Finally, Proposition \ref{killing all algebras} yields a field extension $F$ of $K$ such that the map $\Res_{F/K}\colon H^1(K)\to H^1(F)$ is injective and $\Res_{F/K}\colon f_2(W)\to H^2(F)$ is an isomorphism. 
Then the composition  
\[
(\widehat f_1,\widehat f _2)=(\Res_{F/K}\circ f_1,\Res_{F/K}\circ f_2)\colon (V,W,b,\epsilon)\to (H^1(F),H^2(F),\cup,(-1)_F)
\]
is a monomorphism of $\mathbb{F}_p$-bilinear maps, and moreover, $\widehat f_2$ is an isomorphism. 
\end{proof}

We deduce an analogous result for non-augmented $\mathbb{F}_p$-bilinear maps, provided that they are skew-symmetric:

\begin{cor}
\label{main result for skew-symmetric bilinear maps}
Let $V,W$ be $\mathbb{F}_p$-linear spaces and let $b\colon V\times V\to W$ be a skew-symmetric $\mathbb{F}_p$-bilinear map.
Then there exists a field $F$ containing a root of unity of order $p$, an $\mathbb{F}_p$-linear monomorphism $f_1\colon V\to H^1(F)$ and an $\mathbb{F}_p$-linear isomorphism $f_2\colon W\to H^2(F)$, such that the following diagram commutes
\[
\xymatrix{
H^1(F)&*-<2pc>{\times}&H^1(F)\ar[r]^{\cup}&H^2(F)\\
\strut V\ar@{>->}[u]^{f_1}&*-<2pc>{\times}&\strut V\ar@{>->}[u]_{f_1}\ar[r]^{b}&W\ar[u]^{\wr}_{f_2}.
}
\]
\end{cor}
\begin{proof}
Lemma \ref{bilinear extends to augmented bilinear} yields an augmented $\mathbb{F}_p$-bilinear map $(\widehat V,W,\widehat b,\epsilon)$ such that $V$ is an $\mathbb{F}_p$-subspace of $\widehat V$ and $\widehat b$ extends $b$.
Now apply Theorem \ref{precise main theorem} for this augmented bilinear map to obtain a field $F$ and $\mathbb{F}_p$-linear maps $f_1,f_2$ as required.
\end{proof}

\section{Graphs}
\label{section on graphs}
Our next aim is to show that, in general, the monomorphism in Theorem \ref{precise main theorem} cannot be required to be an isomorphism even for surjective augmented bilinear maps.
Examples showing this will be given in the next section.
They will be based on a general construction of augmented bilinear maps from graphs, which we now describe.

\begin{defin}
\rm
A \textsl{finite simplicial graph} (or \textsl{na\"ive graph}) $\Gamma=(V(\Gamma),E(\Gamma))$ is a finite unoriented graph $\Gamma$ with no loops and no double edges, where
$V(\Gamma)=\{v_1,\ldots,v_n\}$ is its set of vertices, and $E(\Gamma)$ is its set of edges. 
\end{defin}
We further say that a simplicial graph $\Gamma_0$ is an \textsl{induced subgraph} of the graph $\Gamma$, if $V(\Gamma_0)\subseteq V(\Gamma)$ and $E(\Gamma_0)$ consists exactly of the edges in $E(\Gamma)$ connecting vertices in $V(\Gamma_0)$.

\begin{exam}
\rm
Let $n\geq3$.
We write $L_{n-1}$, resp., $C_n$, for the simplicial graphs consisting of $n$ vertices forming a line, resp., a circle:
\[
\xymatrix@R=5pt{\\
\bullet\ar@{-}[r]&\bullet\ar@{-}[r]&\cdots\ar@{-}[r]&\bullet\ar@{-}[r]&\bullet  &
\bullet\ar@/^0.5cm/@{-}[rrrr]\ar@{-}[r]&\bullet\ar@{-}[r]&\cdots\ar@{-}[r]&\bullet\ar@{-}[r]&\bullet. \\
&&L_{n-1}&&&&&C_n
}
\]
We note that for $n\geq4$, any of the graphs $L_{n-1}$, $C_n$ contains one of $L_3$ or $C_4$ as an induced subgraph.
\end{exam}

We recall from e.g., \cite{CassellaQuadrelli21} or \cite{SnopceZalesskii22} the following group-theoretic construction:

\begin{defin}
\rm
Consider a finite simplicial graph $\Gamma=(V(\Gamma),E(\Gamma))$.
The \textsl{pro-$p$ right-angled Artin group} (abbreviated \textsl{RAAG}) associated with $\Gamma$ is the pro-$p$ group $G_\Gamma$ with $V(\Gamma)$ as a set of generators, subject to the requirement that vertices $v,v'$ commute whenever there is an edge connecting them.
Thus $G_\Gamma$ has the pro-$p$ presentation
\[
G_\Gamma=\Bigl\langle V(\Gamma)\ \Bigm|\ [v,v']=1 \hbox{ if and only if }\{v,v'\}\in E(\Gamma)\Bigr\rangle_{{\rm pro-}p}. 
\] 
\end{defin}

\begin{exam}
\rm
One has 
\[
\begin{split}
G_{L_3}&=\bigl\langle v_1,v_2,v_3,v_4\ \bigm|\  [v_1,v_2]=[v_2,v_3]=[v_3,v_4]=1\bigr\rangle_{{\rm pro-}p}, \\
G_{C_4}&=\bigl\langle v_1,v_2,v_3,v_4\ \bigm|\  [v_1,v_2]=[v_2,v_3]=[v_3,v_4]=[v_4,v_1]=1\bigr\rangle_{{\rm pro-}p}.
\end{split}
\]
Thus $G_{C_4}$ is the direct product of two copies of the free pro-$p$ groups on two generators, namely $\langle v_1,v_3\rangle_{{\rm pro-}p}$ and $\langle v_2,v_4\rangle_{{\rm pro-}p}$.
\end{exam}

Given a finite graph $\Gamma$ with vertices $v_1,\ldots,v_n$, we let $V=V_\Gamma$, $W=W_\Gamma$ be the $\mathbb{F}_p$-linear spaces on the bases $V(\Gamma)$, $E(\Gamma)$, respectively.
Denote the coordinate of $w\in W$ at $e\in E(\Gamma)$ by $w_e$.
We define the $\mathbb{F}_p$-bilinear map $b_\Gamma\colon V\times V\to W$ by setting for $1\leq k<l\leq n$ and $e\in E(\Gamma)$:
\[
b_\Gamma(v_k,v_l)_e=
\begin{cases}
1,&\hbox{if } (v_k,v_l)=e,\\
0,&\hbox{otherwise.}
\end{cases}
\]
Then $(V_\Gamma,W_\Gamma,b_\Gamma,0)$ is a surjective augmented $\mathbb{F}_p$-bilinear map.
We call it the \textsl{augmented bilinear map associated with $\Gamma$}.

\begin{prop}
\label{cup product for RAAGs}
For a simplicial graph $\Gamma$, there is an isomorphism of augmented $\mathbb{F}_p$-bilinear maps 
\[
(H^1(G_\Gamma),H^2(G_\Gamma),\cup,0)\isom(V_\Gamma,W_\Gamma,b_\Gamma,0).
\]
\end{prop}

For the proof we first recall from \cite{Koch02}*{\S7.8} and \cite{NeukirchSchmidtWingberg}*{Ch.\ III, \S3} the duality between the cup product and the relation structure of a general finitely generated pro-$p$ group.

Let $S$ be free pro-$p$ group on finitely many generators $x_1,\ldots,x_n$.
We write $S^{(t)}$, $t=1,2,\ldots$, for the \textsl{lower $p$-central filtration} of $S$, defined inductively (in the pro-$p$ sense) by
\[
S^{(1)}=S, \quad S^{(t+1)}=(S^{(t)})^p[S,S^{(t)}].
\]
We identify $H^1(S)$ with the group of all profinite group homomorphisms $S\to\mathbb{F}_p$.
The substitution map gives a bilinear map $S/S^{(2)}\times H^1(S)\to\mathbb{F}_p$, which is \textsl{perfect}, in the sense that both its left and right kernels are trivial.
Let $\chi_1,\ldots,\chi_n$ be an $\mathbb{F}_p$-linear basis of $H^1(S)$ which is dual to the cosets of $x_1,\ldots x_n$ in $S/S^{(2)}$.

Let $R$ be a closed normal subgroup of $S$ which is contained in the pro-$p$ Frattini subgroup $S^{(2)}=S^p[S,S]$, and let $G=S/R$.
We thus have a minimal pro-$p$ presentation
\[
1\to R\to S\to G\to 1
\] 
of $G$.
The inflation map $\Inf\colon H^1(G)\to H^1(S)$ is an isomorphism, and as $S$ is free pro-$p$, $H^2(S)=0$ \cite{NeukirchSchmidtWingberg}*{Prop.\ 3.5.17}.
It therefore follows from the five term exact sequence in profinite cohomology \cite{NeukirchSchmidtWingberg}*{Prop.\ 1.6.7} that the transgression map $\trg\colon H^1(R)^G\to H^2(G)$ is an isomorphism.

There is a well-defined perfect $\mathbb{F}_p$-bilinear map \cite{EfratMinac11}*{Cor.\ 2.2}
\[
R/R^p[S,R]\times H^1(R)^G\to\mathbb{F}_p, \qquad (\bar r,\psi)\mapsto\psi(r).
\]
It gives rise to a perfect $\mathbb{F}_p$-bilinear map
\[
\langle\cdot,\cdot\rangle\colon R/R^p[S,R]\times H^2(G)\to\mathbb{F}_p, \qquad \langle\bar r,\omega\rangle\mapsto-(\trg^{-1}(\omega))(r).
\]
We call it the \textsl{transgression pairing}.

Let $r_1,\ldots,r_m$ be a minimal system of generators of $R$ as a closed normal subgroup of $S$.
Equivalently, the cosets of $r_1,\ldots,r_m$ form an $\mathbb{F}_p$-linear basis of $R/R^p[S,R]$ \cite{NeukirchSchmidtWingberg}*{Cor.\ 3.9.3}.
By duality, there is an isomorphism
\[
H^2(G)\xrightarrow{\sim}\mathbb{F}_p^m, \quad
\omega\mapsto \bigl(\langle \bar r_i,\omega\rangle\bigr)_{1\leq i\leq m}.
\]
 
Every element $r$ of $R$ has a unique presentation modulo  $S^{(3)}$ as
\[
r\equiv  \prod_{i=1}^nx_i^{pa_{ri}}\prod_{1\leq i<j\leq n}[x_i,x_j]^{a_{rij}}\pmod{S^{(3)}},    
\]
with $a_{ri},a_{rij}\in \{0,1,\ldots p-1\}$.
Viewing $\chi_1.\ldots,\chi_n$ also as elements of $H^1(G)$, one has for $r\in R$ and  $i<j$ that
\begin{equation}
\label{relations and cup products}
a_{rij}=\langle \bar r,\chi_i\cup\chi_j\rangle
\end{equation}
\cite{NeukirchSchmidtWingberg}*{Prop.\ 3.9.13}.

\begin{rem}
\rm
In a similar manner, the exponents $a_{ri}$ are related to the images of the $\chi_i$ under the Bockstein homomorphism $H^1(G)\to H^2(G)$ \cite{NeukirchSchmidtWingberg}*{Prop.\ 3.9.14}.
These connections can be extended to higher terms $S^{(t)}$ in the lower $p$-central filtration using tools from the combinatorics of words \cite{Efrat17}.
We refer to \cite{Efrat23} for the analogous theory in the case of the $p$-Zassenhaus filtration.
\end{rem}

\begin{proof}[Proof of Proposition \ref{cup product for RAAGs}]
We denote the vertices of $\Gamma=(V(\Gamma),E(\Gamma))$  by $v_1,\ldots,v_n$.
Let again $S$ be the free pro-$p$ group on the basis $x_1,\ldots,x_n$, take $G=G_\Gamma$, and let $R$ be the kernel
of the epimorphism $S\to G_\Gamma$, $x_i\mapsto v_i$.
Since $R$ is generated as a closed normal subgroup of $S$ by commutators, it is contained in $S^{(2)}$, and the previous discussion applies.

The commutators $r_e=[x_k,x_l]$ associated with the edges $e=(v_k,v_l)\in E(\Gamma)$, $k<l$, form a minimal system of generators of $R$ as a closed normal subgroup of $S$.
We may therefore identify $H^2(G_\Gamma)$ with $\mathbb{F}_p^{E(\Gamma)}$ as above, with $\omega_e\in\mathbb{F}_p$ denoting the $e$-th coordinate of $\omega\in H^2(G_\Gamma)$. 

For $e=(v_k,v_l)\in E(\Gamma)$, $k<l$, and for $1\leq i<j\leq n$ we have $a_{r_eij}=1$, if $k=i$ and $l=j$, and  $a_{r_eij}=0$ in all other cases.
In view of (\ref{relations and cup products}), the dual basis $\chi_1,\ldots,\chi_n$ of $H^1(G)$ satisfies
\[
(\chi_i\cup \chi_j)_e=
\begin{cases}
1,& \hbox{if }(v_i,v_j)=e,\\
0,&\hbox{otherwise.}
\end{cases}
\]
Hence the bilinear maps $\cup\colon H^1(G_\Gamma)\times H^1(G_\Gamma)\to H^2(G_\Gamma)$ and $b_\Gamma\colon V_\Gamma\times V_\Gamma\to W_\Gamma$ coincide on basis elements, giving the required isomorphism.
\end{proof}

\section{The common slot property}
\label{section on the common slot property}
In \cite{Quadrelli14}*{Th.\ 5.6} Quadrelli proves that the pro-$p$ right-angled Artin group $G_{C_4}$ is not realizable as the maximal pro-$p$ Galois group $G_F(p)$ of any field $F$ containing a root of unity of order $p$.
This result was extended by Snopce and Zalesskii, who used profinite Bass--Serre theory to prove the following remarkable fact:

\begin{thm}[Snopce--Zalesskii \cite{SnopceZalesskii22}]
\label{Snopce Zalesskii}
For every prime $p$ and for a finite simplicial graph $\Gamma$, the pro-$p$ group $G_\Gamma$ is realizable as the maximal pro-$p$ Galois group $G_F(p)$ of a field $F$ containing a root of unity of order $p$ if and only if $\Gamma$ does not contain either of the graphs $L_3$ or $C_4$ as an induced subgraph.
\end{thm}

Several other equivalent Galois-theoretic conditions are given in \cite{SnopceZalesskii22}, as well as in the  earlier work \cite{CassellaQuadrelli21} by Cassella and Quadrelli.
For instance, this is also equivalent to the Bloch--Kato property for $G_\Gamma$, as well as to the universal Koszulity of $H^\bullet(G_\Gamma)$ (see \cite{CassellaQuadrelli21} for definitions).
Another equivalent condition is that $G_\Gamma$ is of \textsl{elementary type} in the following strong sense:
It can be constructed in finitely many steps from the group $\mathbb{Z}_p$ using free pro-$p$ products and direct products with $\mathbb{Z}_p$.  

See \cite{BlumerQuadrelliWeigel24} for a generalization of Theorem \ref{Snopce Zalesskii} to oriented graphs.

\begin{rem}
\label{RAAG and -1}
\rm
As pointed out in \cite{BlumerQuadrelliWeigel24}, when $p=2$ the field $F$ in Theorem \ref{Snopce Zalesskii} can be taken to contain $\sqrt{-1}$.
In fact, using the strong elementary type structure of $G_\Gamma$ one can even show that, for every $p$, there exists such a field $F$ containing all roots of unity of $p$-power order.
However this fact will not be used in the following discussion.
\end{rem}

We now relate Theorem \ref{Snopce Zalesskii} to the augmented bilinear map $(V_\Gamma,W_\Gamma,b_\Gamma,0)$ associated with the graph $\Gamma$.
Namely, we show that for $p=2$, the equivalent conditions of this theorem are also equivalent to $(V_\Gamma,W_\Gamma,b_\Gamma,0)$ being isomorphic to an augmented $\mathbb{F}_2$-bilinear map of field type.
This gives plenty of examples of augmented bilinear maps which are not of field type (up to isomorphism), as desired -- see Example \ref{counterexamples} below.

Our proof is based on the following notion:

\begin{defin}
\label{common slot property}
\rm
Let $V$ and $W$ be sets (at this point not assumed to be linear spaces) and consider a map $b\colon V\times V\to W$.
We say that $b$ has the  \textsl{common slot property} if for every $v,u,v',u'\in V$ with  $b(v,u)=b(v',u')$ there exists $u''\in V$ such that
$b(v,u)=b(v,u'')$ and $b(v',u')=b(v',u'')$.
\end{defin}

\begin{exam}
\label{common slot property for fields}
\rm
Let $p=2$ and let $F$ be a field of characteristic $\neq2$.
Then the cup product map $\cup\colon H^1(F)\times H^1(F)\to H^2(F)$ has the common slot property.
See \cite{Lam05}*{Ch.\ III, Th.\ 4.13} or  \cite{ElmanKarpenkoMerkurjev08}*{Lemma 98.15}, and recall the identification of the cup product $(a_1)_F\cup(a_2)_F$ in $H^2(F)$ with the quaternion algebra $(a_1,a_2/F)$ in the $2$-torsion subgroup ${}_2\mathrm{Br}(F)$ of the Brauer group.
\end{exam}

\begin{thm}
\label{equivalence}
Let $p=2$ and let $\Gamma$ be a finite simplicial graph.
The following conditions are equivalent:
\begin{enumerate}
\item[(a)]
There exists a field $F$ such that
\[
(V_\Gamma,W_\Gamma,b_\Gamma,0)\isom(H^1(F),H^2(F),\cup,(-1)_F)
\]
as augmented $\mathbb{F}_2$-bilinear pairs;
\item[(b)]
The bilinear map $b_\Gamma\colon V_\Gamma\times V_\Gamma\to W_\Gamma$ has the common slot property;
\item[(c)]
There exists a field $F$ such that $G_\Gamma\isom G_F(2)$ and $-1\in (F^\times)^2$;
\item[(d)]
There exists a field $F$ such that $(H^\bullet(G_\Gamma),0)\isom (H^\bullet(F),(-1)_F)$ as $\kappa$-algebras;
\item[(e)]
$\Gamma$ does not contain either of the graphs $L_3$ or $C_4$ as an induced subgraph. 
\end{enumerate}
\end{thm}
\begin{proof}
For a field $F$ of characteristic $2$, the group $G_F(2)$ is free pro-$2$ and $H^i(F)=0$ for every $i\geq2$ \cite{NeukirchSchmidtWingberg}*{Th.\ 6.1.4 and Cor.\ 3.9.5}.
Further, one has $H^i(G_\Gamma)=0$ for every $i\geq2$ if and only if $G_\Gamma$ is a free pro-$p$ group. 
Therefore, when ${\rm Char}\, F=2$, the isomorphisms in (a), (c) and (d) mean that $\Gamma$ has no edges.
But then conditions (b) and (e) hold trivially.
Consequently, we may restrict ourselves in conditions (a), (c) and (d) to fields $F$ of characteristic $\neq2$.
Thus $-1$ is a root of unity of order $2$ in $F$.

\medskip 

(a)$\Rightarrow$(b): \quad
This is a consequence of Example \ref{common slot property for fields}.

\medskip

(c)$\Rightarrow$(d): \quad
Immediate.

\medskip

(d)$\Rightarrow$(a): \quad
We apply the functor ${\bf F}$ of \S\ref{section on kappa-algebras} and use Theorem \ref{cup product for RAAGs}.

\medskip

(c)$\Leftrightarrow$(e): \quad
This follows from Theorem \ref{Snopce Zalesskii} for $p=2$ and Remark \ref{RAAG and -1}.

\medskip

(b)$\Rightarrow$(e): \quad
Arguing by contradiction, we assume that the subgraph of $\Gamma$ induced by the vertices $v_1,v_2,v_3,v_4$ is either $L_3$ or $C_4$.
Thus $\Gamma$ contains the edges $e_1=(v_1,v_2)$, $e_2=(v_2,v_3)$, $e_3=(v_3,v_4)$, and no other edges among $v_1,v_2,v_3,v_4$, except for $e=(v_1,v_4)$ when the induced subgraph is $C_4$.
We may assume that $e_1,e_2,e_3$ correspond to the first three coordinates of $W_\Gamma$.

We claim that the common slot property fails for
\[
v=v_2+v_3,\  u=v_1+v_2, \ v'=v_1+v_2+v_3+v_4,\  u'=v_2.
\]
This will yield the desired contradiction.
Indeed, one has 
\[
b_\Gamma(v,u)=(1,1,0,0,\ldots, 0)=b_\Gamma(v',u').
\]
Now let $a_1,\ldots,a_n\in\mathbb{F}_2$.
The coordinates at $e_1,e_2,e_3$ of the vectors 
\[
b_\Gamma\bigl(v,\sum_{i=1}^na_iv_i\bigr), \text{\quad resp.,\quad} b_\Gamma\bigl(v',\sum_{i=1}^na_iv_i\bigr)
\]
are 
\[
a_1,a_2+a_3,a_4, \text{\quad  resp., \quad} a_1+a_2,a_2+a_3,a_3+a_4.
\] 
However, the system of equations
\[
\begin{cases}
a_1=a_1+a_2=1\\
\qquad\ a_2+a_3=1\\
a_4=a_3+a_4=0
\end{cases}
\]
obviously has no solutions.
Therefore the above two vectors cannot be simultaneously equal to $(1,1,0,0,\ldots, 0)$, proving our claim.
\end{proof}

Theorem \ref{equivalence} provides many examples showing that the restriction to monomorphisms in the Main Theorem is necessary, even for surjective augmented bilinear maps, as these cannot be replaced by isomorphisms:

\begin{exam}
\label{counterexamples}
\rm
Let $\Gamma$ be either $L_{n-1}$ or $C_n$ with $n\geq4$.
Then the augmented $\mathbb{F}_2$-bilinear map $(V_\Gamma,W_\Gamma,b_\Gamma,0)$ is not isomorphic to an augmented $\mathbb{F}_2$-bilinear map of field type.
\end{exam}

In view of these counterexamples, it would make sense to add the common slot property as an additional requirement to the definition of surjective augmented bilinear maps, in the case where the ground ring is $\mathbb{F}_2$.
This is closely related to the following notion due to Marshall (\cite{Marshall80}, \cite{Marshall04}*{\S2}):

\begin{defin}
\label{quaternionic maps}
\rm
A \textsl{quaternionic map} is a surjective map $q\colon G\times G\to Q$, where $G$ is a group of exponent $2$ with a distinguished element $-1$, and $Q$ is a set with a distinguished element $1$, satisfying the following conditions for all $a,b,c\in G$:
\begin{enumerate}
\item[(1)]
$q(a,-a)=1$;
\item[(2)]
$q(a,b)=q(a,c)$ implies that $q(a,bc)=1$;
\item[(3)]
$q(a,b)=q(b,a)$;
\item[(4)]
$q$ satisfies the common slot property.
\end{enumerate}
\end{defin}
To a field $F$ of characteristic $\neq2$ we associate a quaternionic map by taking $G=H^1(F)$, $-1=(-1)_F$, and by taking $Q$ to be the image of the cup product map $q=\cup\colon H^1(F)\times H^1(F)\to H^2(F)$, with $1\in Q$ being the zero element of $H^2(F)$.
It is a long-standing open question in quadratic form theory whether every quaternionic map can be realized as the quaternionic map of a field of characteristic $\neq2$ (see \cite{Marshall04}*{Open Problem 1 and Th.\ 2.2}, \cite{Lam05}*{Ch.\ XII, \S8}).
Note however that this possible axiomatization is restricted to the case $p=2$.

\end{document}